\documentclass[letterpaper,11pt]{amsart}

\usepackage[margin=1.2in]{geometry}
\usepackage{amsmath,amsthm,amssymb}
\usepackage{xspace,xcolor}
\usepackage[breaklinks,colorlinks,citecolor=teal,linkcolor=teal,urlcolor=teal,pagebackref,hyperindex]{hyperref}
\usepackage[alphabetic]{amsrefs}
\usepackage[all]{xy}

\theoremstyle{plain}
\newtheorem{thm}{Theorem}[section]

\newtheorem{lem}[thm]{Lemma}
\newtheorem{prop}[thm]{Proposition}
\newtheorem{cor}[thm]{Corollary}

\theoremstyle{definition}
\newtheorem{defi}[thm]{Definition}

\theoremstyle{remark}
\newtheorem{rmk}[thm]{Remark}

\newtheorem*{ack}{Acknowledgments}

\numberwithin{equation}{section}

\def\N{{\mathbb N}}
\def\Z{{\mathbb Z}}
\def\Q{{\mathbb Q}}
\def\R{{\mathbb R}}
\def\C{{\mathbb C}}

\def\A{{\mathbb A}}
\def\P{{\mathbb P}}

\def\cH{\mathcal{H}}

\def\cV{\mathcal{V}}

\def\I{\mathcal{I}}
\def\J{\mathcal{J}}
\def\O{\mathcal{O}}

\def\fm{\mathfrak{m}}

\def\fq{\mathfrak{q}}

\def\a{\alpha}

\def\d{\delta}
\def\f{\phi}

\def\l{\lambda}
\def\n{\nu}
\def\m{\mu}
\def\om{\omega}
\def\p{\pi}

\def\s{\sigma}
\def\t{\tau}
\def\x{\xi}

\def\S{\Sigma}
\def\Om{\Omega}

\def\.{\cdot}
\def\^{\widehat}
\def\~{\widetilde}

\def\ov{\overline}

\def\rat{\dashrightarrow}
\def\surj{\twoheadrightarrow}
\def\inj{\hookrightarrow}

\def\de{\partial}
\def\lru{\lceil}
\def\rru{\rceil}

\def\({\left(}
\def\){\right)}

\newcommand{\ru}[1]{\lru{#1}\rru}

\renewcommand{\and}{ \ \ \text{ and } \ \ }
\renewcommand{\for}{ \ \ \text{ for } \ \ }

\def\Jac{\mathrm{Jac}}

\DeclareMathOperator{\rk} {rk}

\DeclareMathOperator{\Proj} {Proj}

\DeclareMathOperator{\Pic} {Pic}
\DeclareMathOperator{\Aut} {Aut}
\DeclareMathOperator{\Bir} {Bir}

\DeclareMathOperator{\Sing} {Sing}

\DeclareMathOperator{\Bl} {Bl}
\DeclareMathOperator{\val} {val}

\DeclareMathOperator{\Ex} {Ex}
\DeclareMathOperator{\mult} {mult}
\DeclareMathOperator{\ord} {ord}

\DeclareMathOperator{\Cl} {Cl}
\DeclareMathOperator{\lct} {lct}

\DeclareMathOperator{\Fitt} {Fitt}

\DeclareMathOperator{\mld} {mld}
\DeclareMathOperator{\can} {ct}

\begin{document}

\title{Birational rigidity of singular Fano hypersurfaces}

\author{Tommaso de Fernex}

\address{Department of Mathematics, University of Utah, 155 South 1400
East, Salt Lake City, UT 48112-0090, USA}
\email{{\tt defernex@math.utah.edu}}

\thanks{{\it Key words and phrases.}
Fano hypersurface, birational rigidity, inversion of adjunction, Mather discrepancy.}

\thanks{The research was partially supported by 
NSF grant DMS-1402907 and NSF FRG grant DMS-1265285.}

\thanks{Compiled on \today. Filename {\tt \jobname}}

\begin{abstract}
We establish birational superrigidity for a large class of singular projective Fano hypersurfaces of index one. In the special case of isolated singularities, our result applies for instance to: (1) hypersurfaces with semi-homogeneous singularities of multiplicity asymptotically bounded by twice the square root of the dimension of the hypersurface, (2) hypersurfaces with isolated singularities whose Tyurina numbers satisfy a similar bound, and (3) hypersurfaces with isolated singularities whose dual variety is a hypersurface of degree sufficiently close to the expected degree. 
\end{abstract}

\maketitle

\section{Introduction}

The interest in birationally rigidity originates from the
realization that, differently from the surface case, 
higher dimensional Fano varieties and Mori fiber spaces
present a wide spectrum of possible birational characteristics, 
with rational varieties at one end of the spectrum and
birationally superrigid varieties at the other end.
The problem of determining birational links between different Mori fiber spaces
finds its motivation in the minimal model program, and can be viewed
as the counterpart of the question asking about the existence of flops
between minimal models. 

Birational rigidity has been extensively studied in dimension three, 
and several examples of birationally rigid Fano manifolds
are also known in higher dimensions. 
Starting with Iskovskikh and Manin's theorem on smooth quartic threefolds, 
the case of smooth hypersurfaces of projective spaces has been 
studied and progressively understood, over the arc of forty years, 
in the papers \cite{IM71,Puk87,Puk98,Che00,Puk02a,dFEM03,dF}, 
culminating with the following theorem. 

\begin{thm}[\protect{\cite[Theorem~A]{dF}}]
\label{t:dF}
For $N \ge 4$, every smooth hypersurface $V$ of degree $N$ in $\P^N$
is birationally superrigid.
\end{thm}

This means that there are no birational 
modifications of $V$ into Mori fiber spaces other than isomorphisms, and it implies that $V$ is not rational.
Since no other smooth Fano hypersurface is birationally superrigid, 
one obtains from this fact the complete list of smooth birationally superrigid Fano hypersurfaces.
Actually, the proof in \cite{dF} has a gap, 
and the main result of the present paper (see Theorem~\ref{t:main} below) 
provides a new proof which works, in the smooth case, 
for all $N \ge 7$, the lower dimensional cases already being 
established in the earlier papers on the subject cited above.\footnote{An {\it erratum} with an amended 
proof in the smooth case has been written to accompany \cite{dF}.}

The main purpose of this paper is to extend this study to singular hypersurfaces, 
a setting that is still far from being understood. 

The property of birational rigidity is quite sensitive to the singularities. 
For example, smooth quartic threefolds are birationally
superrigid, but those with a double point are only birationally rigid
since the projection from the point induces a birational automorphism.
Furthermore, quartic threefolds that are singular (with multiplicity two) along a line
can be birationally modified into conic bundles. 

In low dimensions, there are sporadic results on
the birational rigidity of quartic threefolds and sextic 
fivefolds with mild singularities (mostly ordinary double points)
obtained in \cite{Puk88,CM04,Mel04,Che07}.
A contribution in higher dimensions was given by 
Pukhlikov in \cite{Puk02b,Puk03}, where hypersurfaces with 
semi-homogeneous singularities
are studied under a certain ``regularity'' condition requiring that, 
at each point of the variety, the intermediate homogeneous terms
of the local equation of the hypersurface form a regular sequence.
We recall that semi-homogeneous singularities (also known as 
\emph{ordinary multiple points}) are isolated hypersurface 
singularities whose tangent cone is smooth away from the vertex.

Singular Fano hypersurfaces provide a rich setting to explore. The works on quartic threefolds
show that, in low dimensions, the problem becomes rather delicate already when
dealing with very mild singularities. 
The main result of this paper should be viewed as complementing those studies,
by showing that the situation stabilizes in the strongest possible terms when the 
dimension is let grow and the ``depth'' of the singularities is
maintained, in some suitable sense, asymptotically bounded
in terms of the dimension.

We allow positive dimensional singularities, 
and avoid to impose any ``regularity'' conditions on the local equations of 
the hypersurface. 
The following defines the type of condition on singularities we consider. 

\begin{defi}
\label{def}
Let $V \subset \P^N$ be a hypersurface, and let $P \in V$ be a closed point. 
For any pair of integers $(\d,\n)$ with $\d \ge -1$ and $\n \ge 1$, 
we say that $P$ is a \emph{singularity of type $(\d,\n)$} if 
the singular locus of $V$ has dimension at most $\d$ near $P$ and
given a general complete intersection $X \subset V$ 
of codimension $c = \min\{\d+2,\dim V\}$ through $P$, 
the $(\n-1)$-th power of the maximal ideal $\fm_{X,P} \subset \O_X$ is contained
in the integral closure of the Jacobian ideal $\Jac_X$ of $X$.
\end{defi}

For instance, regular points are singularities of type $(-1,1)$, and 
semi-homogeneous hypersurface singularities 
of multiplicity $\n$ are singularities of type $(0,\n)$.
More generally, every isolated hypersurface singularity of multiplicity $\n$ whose 
tangent cone is smooth away from a set of dimension two
is a singularity of type $(0,\n)$.
In general, singularities of type $(\d,\n)$ are also of type $(\d',\n')$ for every
$\d' \ge \d$ and $\n' \ge \n$. 

We can now state our main result. 

\begin{thm}
\label{t:main}
Let $N$, $\d$ and $\n$ be fixed integers with $\d \ge -1$, $\n \ge 1$, and
\[
2\d+\n + 7 \le \frac{2(N+1)}{\sqrt N}.
\]
Then every hypersurface $V \subset \P^N$ of degree $N$
with only singularities of type $(\d,\n)$ is a 
Fano variety with Picard number 1 and factorial terminal singularities, 
and is birationally superrigid.  
In particular, $V$ is not rational and $\Bir(V) = \Aut(V)$.
\end{thm}

The proof of this theorem
combines the method of maximal singularities
with inversion of adjunction, Nadel's vanishing theorem, 
and properties of Mather log discrepancies. 
Even assuming that there are not singularities, 
the core of the proof is quite different from the original proof given in the smooth case
in \cite{dF}. 

To illustrate Theorem~\ref{t:main} when $V$ is singular, we present three special cases
in which the singularities are isolated. 
In order to keep the formulas in the statements as simple as possible, we 
apply the theorem under the stronger assumption that 
\[
2\d+\n + 7 \le 2\sqrt N.
\]
We start with the case of semi-homogeneous singularities.

\begin{cor}
\label{c:ordinary}
Every hypersurface $V \subset \P^N$ of degree $N$ with semi-homogeneous singularities
of multiplicity at most $2\sqrt N - 7$ is birationally superrigid.
\end{cor}

Comparing this with the results of Pukhlikov, one sees that
while the bounds on multiplicity in the corollary are more restrictive 
than those in his papers, no ``regularity'' 
assumption is required in our result. Furthermore, the hypothesis on the singularities
being semi-homogeneous can be relaxed by allowing, for instance, 
the tangent cones to have singularities in dimension 1 or 2.

Another special case of the theorem can be formulated in terms of 
the Tyurina numbers of the singularities.
Let $\t_P(V)$, $\t'_P(V)$ and $\t''_P(V)$ be, respectively, the Tyurina numbers (at $P$) of
$V$, of a general hypersurface in $V$ passing through $P$, 
and of a general complete intersection of codimension 2 through $P$. 

\begin{cor}
\label{c:Tyurina}
Let $V \subset \P^N$ be a hypersurface of degree $N$ with isolated 
singularities, and assume that for every $P \in V$
\[
\min\{\,\t_P(V),\, \t'_P(V),\, \t''_P(V)\,\}  \le 2\sqrt N - 8.
\]
Then $V$ is birationally superrigid.
\end{cor}

Since the Tyurina number is bounded above by the Milnor number, 
a similar corollary can be formulated in terms of the Milnor numbers
of general restrictions of $V$, 
which are known as the \emph{Teissier--Milnor numbers} of $V$ \cite{Tei73}. 
Using then a result of Teissier \cite{Tei77}, we obtain the following result, 
which comes unexpected to us.

\begin{cor}
\label{c:dual}
Let $V \subset \P^N$ be a hypersurface of degree $N$ with isolated 
singularities, and assume that the dual variety $\check{V} \subset \check\P^N$
is a hypersurface of degree 
\[
\deg\check V \ge N(N-1)^{N-1} - (4\sqrt N + 2s - 18), 
\]
where $s$ is the number of singular points of $V$. 
Then $V$ is birationally superrigid.
\end{cor}

Properties of singularities of type $(\d,\n)$ are discussed in Section~\ref{s:sing},
and the three corollaries above are proven in Section~\ref{s:cor}.
The subsequent section gathers several definitions and properties of singularities
and multiplicites; in order to deal with the singularities of the hypersurface, we 
work with Mather log discrepancies, which are recalled there. 
Finally, the last section is devoted to the proof of Theorem~\ref{t:main}.
All varieties are assumed to be defined over the field of complex numbers $\C$.

\begin{ack}
We thank J\'anos Koll\'ar for several useful 
comments and for pointing out an error in a lemma of \cite{dF} 
which was used in a previous version of this paper. 
We also thank Roi Docampo, Lawrence Ein, Mircea 
Musta\c t\u a, and Fumiaki Suzuki for useful comments and suggestions.
\end{ack}

\section{Singularities of type $(\d,\n)$.}
\label{s:sing}

In this section we discuss some
properties of singularities of type $(\d,\n)$ introduced in Definition~\ref{def}.
For ease of notation, it is convenient to focus on affine hypersurfaces. 
Throughout this section, fix $n \ge 1$, and let $X \subset \A^n$ be a hypersurface. 
Recall that if $h(x_1,\dots,n_n) = 0$ is an equation for $X$, 
then the Jacobian ideal $\Jac_X \subset \O_X$ is cut out, on $X$, 
by the partial derivatives of $h$:
\[
\Jac_X = \Big(\frac{\de h}{\de x_1},\dots,\frac{\de h}{\de x_n}\Big)\.\O_X.
\]

We say that a closed point $P \in X$ is an isolated singularity
if $X$ is smooth in a punctured neighborhood of $P$. Note that this includes
the possibility that $X$ is smooth at $P$. 
For an isolated singularity $P \in X$, we define
\[
\n_P(X) := \min\big\{\, \n \in \Z_{>0}
\mid (\fm_{X,P})^{\n - 1} \subset \ov{\Jac_X} \,\big\},
\]
where the bar in the right-hand side denotes integral closure. 

\begin{rmk}
A closed point $P$ on a normal hypersurface $V \subset \P^N$
is a singularity of type $(\d,\n)$ if and only if the singular locus has dimension at most  $\d$
and $\n_P(X) \le \n$ for a general complete intersection $X \subset V$
of codimension $c = \min\{\d+2,\dim V\}$ through $P$.
\end{rmk}

\begin{prop}
\label{p:restr}
Assume that $n \ge 2$, and let $P \in X$ be an isolated singularity. 
Then for every general hyperplane section $H \subset X$ through $P$
we have 
\[
\n_P(H) \le \n_P(X).
\]
\end{prop}

\begin{proof}
Teissier's Idealistic Bertini Theorem \cite[2.15 Corollary~3]{Tei77} implies that 
$\ov{\Jac_X|_H} = \ov{\Jac_H}$.
By the definition of integral closure, 
there is an inclusion $\ov{\Jac_X}|_H \subset \ov{\Jac_X|_H}$.
Since $\fm_{X,P}|_H = \fm_{H,P}$, the proposition follows. 
\end{proof}

\begin{rmk}
It follows by Proposition~\ref{p:restr}
that a singularity of type $(\d,\n)$ of a hypersurface $V \subset \P^N$ 
is also of type $(\d',\n')$ for every $\d' \ge \d$ and $\n' \ge \n$. 
\end{rmk}

A special case where $\n_P(X)$ is easy to compute is when
$P \in X$ is a semi-homogeneous hypersurface singularity.
We denote by $e_P(X)$ the \emph{multiplicity} of $X$ at $P$, given by 
the degree of the tangent cone $C_PX$. 

\begin{prop}
\label{p:ordinary}
If $P \in X$ is a semi-homogeneous hypersurface singularity, then
\[
\n_P(X) = e_P(X). 
\]
\end{prop}

\begin{proof}
Let for short $m := e_P(X)$. 
Let $f \colon \~X \to X$ and $g\colon \~\A^n \to \A^n$ be the blow-ups of 
$X$ and $\A^n$ at $P$, and let $F$ and $G$ be the respective exceptional divisors. 
Then $\~X \subset \~\A^n$ is the proper transform of $X$ and $g^*X = \~X + m G$.
If $(x_1,\dots,x_n)$ are affine coordinates centered at $P$,
and $h(x_1,\dots,x_n) = 0$ is an equation defining $X$, then
$\mult_P(h) = m$, and thus $\mult_P(\de h/ \de x_i) = m-1$. 
By hypothesis, $F = \~X \cap G$ is a smooth hypersurface of degree $m$ in $G \cong \P^{n-1}$, 
defined by the vanishing of the degree $m$ homogeneous form $h_m$ of $h$. 
It follows that the homogeneous ideal 
\[
\Big(\frac{\de h_m}{\de x_1},\dots,\frac{\de h_m}{\de x_n}\Big) \subset \C[x_1,\dots,x_n]
\]
has no zeroes in $\P^{n-1}$. 
This implies that  $\Jac_X \.\,\O_{\~X} = \O_{\~X}(-(m-1)F)$,
and thus $\ov{\Jac_X} = f_*\O_{\~X}(-(m-1)F)$.
The assertion follows then by the fact that $(\fm_{X,P})^k \.\, \O_{\~X} = \O_{\~X}(-kE)$.
\end{proof}

The Jacobian ideal retains important information of a singularity. 
For instance, it is a theorem of Mather and Yau \cite{MY82} that, 
for an isolated hypersurface singularity $P \in X$,
the Jacobian $\C$-algebra $\O_{X,P}/\Jac_X$ determines 
the analytic isomorphism class of the singularity.
The dimension of this algebra 
is called the \emph{Tyurina number} of the singularity. 
If, as above, $X$ is defined by $h(x_1,\dots,x_n) = 0$ in $\A^n$
and $P = (0,\dots,0)$, then the Tyurina number is given by 
\[
\t_P(X) := \dim_\C \frac{\C[[x_1,\dots,x_n]]}
{\big(h , \frac{\de h}{\de x_1},\dots,\frac{\de h}{\de x_n}\big)}.
\]
The Tyurina number is closely related to the \emph{Milnor number} of the singularity, 
which is the number of spheres in the bouquet homotopically
equivalent to the Milnor fiber and is computed by the dimension
\[
\m_P(X) := \dim_\C \frac{\C[[x_1,\dots,x_n]]}
{\big(\frac{\de h}{\de x_1},\dots,\frac{\de h}{\de x_n}\big)}.
\]
For every $i$, we define the \emph{$i$-th Tyurina number} $\t^{(i)}_P(X)$
and the \emph{$i$-th Teissier--Milnor number} $\m^{(i)}_P(X)$
of $X$ at $P$ to be, respectively, the Tyurina number and the Milnor number 
of a general complete intersection
of codimension $i$ passing through $P$.\footnote{The reader is cautioned that the notation adopted here
differs with the notation originally used by Teissier \cite{Tei73} where the index $i$
refers to the dimension of the projective subspace cutting out the section, 
rather than the codimension of the section in $X$.}
For $i=0,1,2$, we just write
$\t_P(X),\t_P'(X),\t_P''(X)$ and $\m_P(X),\m_P'(X),\m_P''(X)$.

\begin{prop}
\label{p:Tyurina}
With the above notation, we have
\[
\n_P(X) \le \t_P(X) + 1
\] 
\end{prop}

\begin{proof}
Let for short $\n := \n_P(X)$.
By definition, we have $(\fm_{X,P})^{\n - 2} \not\subset \ov{\Jac_X}$.
In view of the valuative interpretation of integral closure, 
this means that there is a divisorial valuation $v$ on the function field of $X$, 
with center $P$, such that
\[
(\n - 2)\. v(\fm_{X,P}) < v(\Jac_X).
\]
Consider the sequence of ideals $\fq_k := (\fm_{X,P})^k + \Jac_X \subset \O_X$. 
Since $v(\fq_k) = k\.v(\fm_{X,P})$ for $1 \le k \le \n - 2$, we have
a chain of strict inclusions of ideals
\[
\O_X \supsetneq \fq_1 \supsetneq \fq_2 \supsetneq \dots 
\supsetneq \fq_{\n-2} \supsetneq \Jac_X.
\]
This implies that $\t_P(X) \ge \n - 1$.
\end{proof}

\begin{rmk}
The inequality in Proposition~\ref{p:Tyurina} may look weak at a first glance, and in fact
much stronger inequalities hold in many 
cases (for instance, for semi-homogeneous singularities).
The inequality is however optimal as stated. 
Examples where equality is achieved for each possible value of $\n_P$
are given by the hypersurfaces 
$X_d = (x_1^2 + \dots + x_{n-1}^2 + x_n^d = 0) \subset \A^n$, $d \ge 1$, for which
$\n_P(X_d) = d$ and $\t_P(X_d) = d-1$, $P$ being the origin in $\A^n$.
\end{rmk}

\section{Proofs of the corollaries}
\label{s:cor}

In this short section we prove the three corollaries stated in the introduction.

\begin{proof}[Proof of Corollary~\ref{c:ordinary}]
By Proposition~\ref{p:ordinary}, $P \in V$ is a singularity of type $(0,e_P(X))$
for a general complete intersection $X \subset V$
of codimension two passing through $P$. 
Since $e_P(X) = e_P(V)$, the corollary follows 
directly from Theorem~\ref{t:main}.  
\end{proof}

\begin{proof}[Proof of Corollary~\ref{c:Tyurina}]
Let $P \in V$ be one of the singularities of $V$, and fix $i \in \{0,1,2\}$ 
such that $\t_P^{(i)}(V) \le 2\sqrt N - 8$. 
If $V^{(i)} \subset V$ denotes a general
complete intersection of codimension $i$ through $P$, then we have
$\n_P(V^{(i)}) \le 2\sqrt N - 7$ by Proposition~\ref{p:Tyurina}.
Since $i \le 2$, it follows by Proposition~\ref{p:restr} that if $V'' \subset V$
is a general complete intersection of codimension two then
$\n_P(V'') \le  2\sqrt N - 7$. Then the corollary follows from Theorem~\ref{t:main}.
\end{proof}

\begin{proof}[Proof of Corollary~\ref{c:dual}]
Let $P_1,\dots,P_s \in V$ be the singular points. 
It is proven in \cite[Appendice~II.3]{Tei80} that the dual variety has degree
\[
\deg\check V = N(N-1)^{N-1} - \sum_{j=1}^s (\m_{P_j}(V) + \m'_{P_j}(V)).
\]
It follows by our assumption of the degree of $\check V$ that
\[
\sum_{j=1}^s (\m_{P_j}(V) + \m'_{P_j}(V)) \le 4\sqrt N + 2 s - 18.
\]
Bearing in mind that, for every $j$, both $\m_{P_j}(V)$ and $\m'_{P_j}(V)$ are positive 
integers, we deduce that 
$\m_{P_j}(V) + \m'_{P_j}(V) \le 4\sqrt N -16$ for any given $j$, and hence
\[
\min\{\,\m_{P_j}(V),\, \m'_{P_j}(V)\,\} \le 2\sqrt N - 8.
\]
Since $\t^{(i)}_{P_j}(V) \le \m^{(i)}_{P_j}(V)$, we conclude by Corollary~\ref{c:Tyurina}.
\end{proof}

\section{Log discrepancies and multiplicities}

In this section we review some results related to singularities of pairs and multiplicities. 
General references on the subject are \cite{KM98,Laz04}.

Let $X$ be a variety, and
let $E$ be a prime divisor on a resolution of singularities $f \colon X' \to X$.
We say that $E$ is a \emph{divisor over} $X$;
the image of $E$ in $X$ is called the \emph{center} of $E$.
When $X$ is normal, we say that the divisor $E$ is \emph{exceptional over $X$} if its center
has codimension $\ge 2$ in $X$.

The divisor $E$ defines a valuation $\val_E$ over $X$, with valuation ring $\O_{X',E}$. 
If $Z \subset X$ is a proper closed
subscheme and $I_Z \subset \O_X$ is its ideal sheaf, 
then we set $\val_E(Z) := \val_E(I_Z)$. 
If $Z = \sum c_iZ_i$ if a finite formal $\Q$-linear combination of 
proper closed subschemes $Z_i \subset X$, then we denote
$\val_E(Z) := \sum c_i\val_E(Z_i)$. 

We will use the following basic fact without further notice.
We refer to \cite[Lemma~2.3]{dFM} for a proof.

\begin{lem}
Let $X \to Y$ be a dominant morphism of varieties. 
If $E$ is a divisor over $X$, then the restriction of $\val_E$ to $\C(Y)$
is a valuation of the form $q\val_F$ for some divisor $F$ over $Y$
and some positive integer $q$. 
\end{lem}

We consider \emph{pairs} of the form $(X,Z)$ where $X$ is a variety and
$Z = \sum c_i Z_i$ is a finite, formal 
$\Q$-linear combination of proper closed subschemes $Z_i \subset X$. 
The pair is said to be \emph{effective} if $c_i \ge 0$ for all $i$. 

We say that a variety $X$, or a pair $(X,Z)$, is \emph{$\Q$-Gorenstein} 
if $X$ is normal and the canonical class $K_X$ of $X$ is $\Q$-Cartier.
The \emph{log discrepancy}
of a $\Q$-Gorenstein pair $(X,Z)$ along $E$ is defined to be
\[
a_E(X,Z) := \ord_E(K_{X'/X}) + 1 - \val_E(Z),
\]
where $K_{X'/X}$ is the relative canonical divisor.
If $Z=0$, then we drop it from the notation and write $a_E(X)$.  
A $\Q$-Gorenstein pair $(X,Z)$ is \emph{log canonical} (resp., \emph{log terminal})
if $a_E(X,Z) \ge 0$ (resp., $a_E(X,Z) > 0$) for every prime divisor $E$ over $X$. 
The pair is \emph{canonical} (resp., \emph{terminal})
if $a_E(X,Z) \ge 1$ (resp., $a_E(X,Z) > 1$) for every $E$ exceptional over $X$.

A \emph{log resolution} of a pair $(X,Z)$ is a resolution
$f \colon X' \to X$ such that the exceptional locus $\Ex(f)$ of $f$ 
and each subscheme $f^{-1}Z_i \subset X'$ is a Cartier divisor, and their
sum $\Ex(f) + \sum f^{-1}Z_i$ has simple normal crossing support.
If $Z = \sum c_iZ_i$, then we denote $f^{-1}Z := \sum c_if^{-1}Z_i$. 
If $(X,Z)$ is an effective $\Q$-Gorenstein pair, then one defines the
\emph{multiplier ideal} of $(X,Z)$ to be the ideal sheaf
\[
\J(X,Z) := f_*\O_{X'}(\ru{K_{X'/X} - f^{-1}Z}),
\]
where the round-up in the right-hand side is taken componentwise.
The definition is independent of the choice of resolution.

\begin{thm}
[\protect{\cite[Theorem~9.4.17]{Laz04}}]
\label{t:Nadel}
Let $(X,cZ)$ be an effective $\Q$-Gorenstein pair where $Z$ is a subscheme and $c \ge 0$.
Let $L$ and $A$ be Cartier divisors such that 
$\O_X(A) \otimes I_Z$ is globally generated and
$L - (K_X + cA)$ is nef and big. Then 
\[
H^i(X,\J(X,cZ) \otimes \O_X(L)) = 0 \for i > 0.
\]
\end{thm}

The \emph{minimal log discrepancy} of a $\Q$-Gorenstein pair 
$(X,Z)$ along a proper closed subset $T \subset X$ 
is the infimum of all log discrepancies along divisors with center in $T$, 
and is denoted by $\mld(T;X,Z)$. 
We will use the following inversion of adjunction property. 

\begin{thm}[\protect{\cite[Theorem~1.1]{EM04}}]
\label{t:inv-adj}
Consider an effective pair $(X,Z)$ where 
$X$ is a normal variety with locally complete intersection singularities
and $Z = \sum c_i Z_i$, and let
$Y \subset X$ be a normal, locally complete intersection subvariety of codimension $e$
that is not contained in $\bigcup_i Z_i$. 
Then for every proper closed subset $T \subset Y$ we have 
\[
\mld(T;X,Z+eY) = \mld(T; Y, Z|_Y).
\]
\end{thm}

The \emph{log canonical threshold} of an effective $\Q$-Gorenstein pair $(X,Z)$ is defined by
\[
\lct(X,Z):= \sup \{\, c \in \R_{\ge 0} \mid \text{$(X,cZ)$ is log canonical}\, \},
\]
(where we set $\sup \emptyset = -\infty$).
Note that, for any $c\ge 0$, $\lct(X,Z) > c$ if and only if $\J(X,cZ) = \O_X$. 
We denote by $\lct_P(X,Z)$ the log canonical threshold of $(X,Z)$ at $P$, 
defined as the the minimum of the log canonical thresholds $\lct(U,Z|_U)$
over all open neighborhoods $U$ of $P$. 

In a similar fashion, we
define the \emph{canonical threshold} of an effective $\Q$-Gorenstein pair $(X,Z)$ by
\[
\can(X,Z):= \sup \{\, c \in \R_{\ge 0} \mid \text{$(X,cZ)$ is canonical}\, \}.
\]
Note that $\can(X,Z) > 0$ if $X$ has terminal singularities.
We denote by $\can_P(X,Z)$ the canonical threshold of $(X,Z)$ at $P$.

A \emph{Mori fiber space} is a normal projective variety $X$
with $\Q$-factorial terminal singularities, equipped with an
extermal Mori contraction of fiber type $g \colon X \to S$ 
(so that $\dim S < \dim X$, $g_*\O_X = \O_S$, $\rk\Pic(S) = \rk \Pic(X) - 1$, 
and $-K_X$ is relatively ample over $S$). 
A Mori fiber space is said to be \emph{birationally superrigid}
if there are no birational maps to other Mori fiber spaces
other than isomorphisms.

The following result, known as the Noether--Fano inequality,
is central for the method of maximal singularities. 
The result is essentially due to \cite{IM71}. A proof using the minimal model program
is given in \cite{Cor95}; see also \cite{dF14} for a short, self-contained proof.

\begin{thm}
\label{t:NF}
Let $X$ be a Fano variety of Picard number 1 with
terminal $\Q$-factorial singularities.
Suppose that there is a birational map $\f \colon X \rat X'$ where
$X'$ is a Mori fiber space.
Fix an embedding $X' \subset \P^m$, and let $\cH := \f_*^{-1}|\O_{X'}(1)|$
be the linear system on $X$ giving the map $X \rat X' \inj \P^m$. 
Let $B(\cH) \subset X$ be the base scheme of $\cH$, and
let $r$ be the rational number such that $\cH \subset |-rK_X|$. 
If $\f$ is not an isomorphism, then 
\[
\can(X,B(\cH)) < 1/r.
\]
\end{thm}

We now turn to a variant (and more general) notion of log discrepancy, called
Mather log discrepancy. 
While the usual log discrepancy is defined by comparing canonical divisors, 
this variant is defined by comparing sheaves of
K\"ahler differentials. 

Let $X$ be a variety of dimension $n$.
Let $f \colon X' \to X$ be a resolution of singularities, and let
$\Jac_f := \Fitt^0(\Om_{X'/X}) \subset \O_{X'}$ be the Jacobian ideal of the map. 
For every prime divisor $E$ on $X'$, we define the \emph{Mather log discrepancy} of a pair $(X,Z)$
along a prime divisor $E$ over $X$ to be
\[
\^a_E(X,Z) := \ord_E(\Jac_f) + 1 - \val_E(Z).
\]
If $Z=0$, then we simply write $\^a_E(X)$. 

\begin{rmk}
\label{r:^k_E-lci}
If $X$ has locally complete intersection singularities, then
$\^a_E(X) = a_E(X) + \val_E(\Jac_X)$ (cf.\ \cite[Corollary~3.5]{dFD}).
\end{rmk}

The \emph{minimal Mather log discrepancy} of a pair $(X,Z)$ along a proper closed subset $T \subset X$ 
is the infimum of all Mather log discrepancies along divisors with center in $T$, 
and is denoted by $\^\mld(T;X,Z)$. The reader is cautioned that in general minimal Mather
log dicrepancies do not satisfy an inversion of adjunction theorem analogous to
Theorem~\ref{t:inv-adj}.

\begin{prop}
\label{p:^lct}
If $P \in X$ is a closed point on a variety $X$ of dimension $n$, 
then we have $\^\mld_P(X,nP) \ge 0$.
\end{prop}

\begin{proof}
Let $E$ be an arbitrary divisor over $X$ with center $P$.
Let $\p \colon \A^N \to Y := \A^n$ be a general linear projection, and let 
$Q := \p(P)$. We have $\val_E|_{\C(Y)} = q\val_F$ where
$F$ is a divisor over $Y$ with center $Q$ and $q$ is a positive integer. 
By taking the projection general enough, we can ensure that
\begin{equation}
\label{eq:^lct-1}
\val_E(P) = q\val_F(Q).
\end{equation}
We can assume that there is a diagram
\[
\xymatrix{
X' \ar[d]_g \ar[r]^f & X \ar[d] \ar@{^(->}[r] & \A^N \ar[d]^\p \\
Y' \ar[r] & Y \ar@{=}[r] & \A^n 
}
\]
where $X' \to X$ and $Y' \to Y$ are resolutions such that $E$ is a divisor on $X'$, 
and $F$ is a divisor on $Y'$. Note that $\ord_E(g^*F) = q$ and $\ord_E(K_{X'/Y'}) = q-1$. 
Denoting by $h \colon X' \to Y$ the composition of $f$ with the projection to $Y$, 
we have $\ord_E(K_{X'/Y}) = \val_E(\Jac_h)$. 
If $x_1,\dots,x_n$ are local parameters in $X'$ centered at a general point of $E$, 
then $f$ is locally given by equations $y_i = f_i(x_1,\dots,x_n)$, and 
$\Jac_f$ is locally defined by the $n\times n$ minors of the matrix
$(\de f_i/\de x_j)$. For a linear projection $\p \colon \A^N \to Y = \A^n$, 
$\Jac_h$ is locally defined by a linear combination of 
the $n\times n$ minors of $(\de f_i/\de x_j)$. 
If the projection is general, then so is the linear combination, and we have
$\^a_E(X) = \val_E(K_{X'/Y})+1$.
Writing $K_{X'/Y} = K_{X'/Y'} + g^*K_{Y'/Y}$, we get
\begin{equation}
\label{eq:^lct-2}
\^a_E(X) = \val_E(K_{X'/Y'}) + \val_E(g^*K_{Y'/Y}) + 1 = q\,a_F(Y).
\end{equation}
Since $E$ is an arbitrary divisor over $X$ with center $P$, and $q \ge 1$, 
we deduce from \eqref{eq:^lct-1} and \eqref{eq:^lct-2} that
$\^\mld_P(X,nP) \ge \mld_Q(Y,nQ)$. Then the proposition follows by
observing that, since $Y$ is smooth of dimension $n$, we have $\mld_Q(Y,nQ) = 0$.
\end{proof}

We will use the following result from \cite{dFM}, stated here in a special case.

\begin{thm}[\protect{\cite[Theorem~2.5]{dFM}}]
\label{t:projection}
Let $X \subset \A^N$ be a Cohen--Macaulay variety of dimension $n$, 
and let $E$ be a divisor over $X$. 
Let $Z \subset X$ closed subscheme of pure codimension $k$
whose ideal in $X$ is locally generated by a regular sequence. 
Then let 
\[
\f \colon X \to \A^{n-k+1} 
\]
be the morphism induced by restriction
of a very general linear projection $\s \colon \A^N \to \A^{n-k+1}$,
so that $\f|_Z$ is a proper finite morphism
and $\f_*[Z]$ is a cycle of codimension one in $\A^{n-k+1}$.
Regard $\f_*[Z]$ as a Cartier divisor on $\A^{n-k+1}$. 
Write $\val_E|_{\C(\A^{n-k+1})} = q \val_G$ where $G$ is a divisor over $\A^{n-k+1}$ 
and $q$ is a positive integer. Then, for every $c > 0$ such that $\^a_E(X,cZ) \ge 0$, we have
\[
q\, a_G\(\A^{n-k+1}, \frac{c^k}{k^k}\, \f_*[Z]\) \le \^a_E(X,cZ).
\]
\end{thm}

We end this section by recalling some properties of multiplicities. 
If $Z$ is a scheme and $\x \in Z$ is a 
(non necessarily closed) point, then the \emph{multiplicity} of $Z$ at $\x$ is 
defined to be the 
Hilbert--Samuel multiplicity of the maximal ideal of $\O_{Z,\x}$
and is denoted by $e_\x(Z)$ (cf.\ \cite[Example~4.3.4]{Ful98}).
If $T \subset Z$ is the closure of $\x$, then we also denote 
this multiplicity by $e_T(Z)$.

If $Z$ is pure-dimensional, then
the function $P \mapsto e_P(Z)$ is upper-semicontinuous on closed points
(cf. \cite[Theorem~(4)]{Ben70}), and 
we have $e_T(Z) = \min_{P \in T}e_P(Z)$ for any subvariety $T \subset Z$.
Here the minimum is taken over the closed points $P$ of $T$, and is achieved
for all points of a dense open subset of $T$. 
If $Z$ is a complete intersection subscheme of a variety $X$, and $T \subset Z$
is an irreducible component, then $e_T(Z)$ is the same as the Hilbert--Samuel 
multiplicity of the ideal of $Z$ in $\O_{X,T}$ (cf.~\cite[Exercise 7.1.10(a)]{Ful98}).

The definition of multiplicity extends in a natural way to cycles. 
If $\a = \sum n_i[Z_i]$ is a cycle on a variety $X$
(here each $Z_i$ is a pure-dimensional subscheme of $X$, without embedded points), 
and $T \subset X$ is a subvariety, then we define $e_T(\a) := \sum n_i e_T(Z_i)$,
where we set $e_T(Z_i) = 0$ whenever $T \not\subset Z_i$. This is well defined
(i.e., it does not depend on the way we write the cycle, cf.\ \cite[Example~4.3.4]{Ful98}).

Proofs of the following two basic properties can be found in \cite[Section~8]{dF}.

\begin{prop}\label{prop:mult-canonical}
Let $D$ be an effective $\Q$-divisor on a smooth variety $X$,
and suppose that $a_E(X,D) \le 1$ for some prime divisor $E$
over $X$. If $P$ is any point in the center of $E$ in $X$, then $e_P(D) \ge 1$.
\end{prop}

\begin{prop}
\label{p:restr-mult}
Let $Z$ be a pure-dimensional Cohen--Macaulay subscheme of $\P^n$,  
and let $\cH \subset (\P^n)^\vee$ be a hyperplane.
Then for any general $H \in \cH$ we have 
$e_P(Z \cap H) = e_P(Z)$ for every $P \in Z \cap H$.
\end{prop}

We close this section with the following property, due to Pukhlikov.

\begin{prop}[\protect{\cite[Proposition~5]{Puk02a}}]
\label{prop:Puk}
Let $X \subset \P^{n+1}$ be a hypersurface, and let 
$\a$ be an effective cycle on $X$ of pure codimension $k \le n/2$. 
Assume that $\a \equiv m\,c_1(\O_X(1))^k$ for some $m \in \N$. 
Then $e_S(\a) \le m$ for every closed 
subvariety $S \subseteq X$ of dimension $\dim S \ge k$ not meeting the singular locus of $X$.
In particular, if $d = \dim \Sing (X)$, then we have $e_T(\a) \le m$ for every closed 
subvariety $T \subseteq X$ of dimension $\dim T \ge d + 1 +k$.  
\end{prop}

\begin{rmk}
The statement of \cite[Proposition~5]{Puk02a} is only given for $k < n/2$, 
but the proof can be extended to include the case
$k = n/2$ (cf.~\cite[Remark~4.4]{dFEM03}). 
\end{rmk}

\section{Proof of Theorem~\ref{t:main}}

We start observing the following general property.

\begin{lem}
\label{l:V}
Any normal hypersurface $V \subset \P^N$ whose singular locus has codimension at 
least 4 is factorial. 
\end{lem}

\begin{proof}
If $\dim V \le 3$ then $V$ is smooth and hence factorial.
Assume then that $\dim V \ge 4$.
The hypersurface $W \subset \P^4$ cut out by $V$
on a general linear 4-space $\P^4 \subset \P^N$ is smooth. 
By the Lefschetz hyperplane theorem, both $\Pic(V)$ and
$\Pic(W)$ are generated by the respective hyperplane classes, 
and so the restriction map $\Pic(V) \to \Pic(W)$ is an isomorphism. 
Since $W$ is smooth, the class map $\Pic(W) \to \Cl(W)$ is an isomorphism.  
On the other hand, the restriction of Weil divisors (which is
well-defined in our setting) induces an isomorphism $\Cl(V) \to \Cl(W)$
by an inductive application of \cite[Theorem~1]{RS06}. It follows that $\Pic(V) \to \Cl(V)$
is an isomorphism. 
\end{proof}

Theorem~\ref{t:main}, whose proof is postponed to the end of the section, 
will be deduced from the following theorem.

\begin{thm}
\label{t:bound}
Fix integers $N,\d,\n,r$ such that $\n,r \ge 1$ and 
\[
-1 \le \d \le \frac N2 -3.
\]
Let $V \subset \P^N$ be a normal hypersurface of degree $N$ with a singularity of type $(\d,\n)$
at a point $P$ and with singular locus of dimension at most $\d$. 
Let $B \subset V$ be a proper closed subscheme of
codimension at least 2, and assume that the sheaf
$\O_V(r)\otimes I_B$ is globally generated. Then
\[
r \can_P(V,B) \ge  \min \left\{ 1 , \, \frac {2(N+1)}{(2\d+\n+7)\sqrt N} \right\}.
\]
\end{thm}

\begin{proof}
Since $N-\d \ge 5$, $V$ is factorial by Lemma~\ref{l:V}.

After replacing $B$ with the intersection of two general members of 
$D, D' \in |\O_V(r)\otimes I_B|$, we reduce to prove the theorem when $B = D \cap D'$
is a codimension 2 complete intersection subscheme of $V$, cut out by two divisors
$D,D' \in |\O_V(r)|$. We denote
\[
c := \can_P(V,B), 
\]
and henceforth assume that $c < 1/r$. 

Note that $N \ge 4$. Since the singular locus of $V$
has at most dimension $\d$, Proposition~\ref{prop:Puk} implies that for every closed
subvariety $T \subset V$ of dimension $\dim T \ge \d+2$ we have $e_T(D) \le r$.
It follows by Proposition~\ref{prop:mult-canonical} that the pair $(V,cB)$
has terminal singularities away from a set of dimension $\d+1$.

We cut down by $\d + 1$ general hyperplanes through $P$. 
Let $\P^{N-\d-1} \subset \P^N$
be a general linear subspace of codimension $\d+1$
passing through $P$, and let 
$W \subset \P^{N-\d-1}$ be the restriction of $V$ to this subspace. 
By inversion of adjunction (Theorem~\ref{t:inv-adj}), $(W,cB|_W)$ is terminal 
away from finitely many points, and is not canonical at $P$.
This implies that $\mld(P;W,cB|_W) \le 1$. 
Adding $P$ to the pair, we get 
\begin{equation}
\label{eq:mld(W)}
\mld(P;W,cB|_W + P) \le 0.
\end{equation}

We take one more hyperplane section. 
Let $\P^{N-\d-2} \subset \P^{N-\d-1}$ be a general hyperplane through $P$, and let 
\[
X \subset \P^{N-\d-2}
\] 
be the restriction of $W$ to $\P^{N-\d-2}$.
We remark that, under our assumption on $\d$ and $N$, we have $\dim X \ge 2$. 
By \eqref{eq:mld(W)} and inversion of adjunction, we have 
\begin{equation}
\label{eq:mld(X)}
\mld(P;X,cB|_X) \le 0.
\end{equation}
Note, on the other hand, that $(X,cB|_X)$ is log terminal in dimension one. 
In fact, we have the following stronger property. 

\begin{lem}
\label{l:1}
The pair $(X,2cB|_X)$ is log terminal in dimension one.
\end{lem}

\begin{proof}
If $N=4$, then $\d=-1$, $X$ is a smooth surface, and $B|_X$ is zero dimensional. Clearly the lemma 
holds in this case. We can therefore assume that $N \ge 5$. 

Let $C \subset X$ be any irreducible curve.

Proposition~\ref{prop:Puk} implies that for every closed
subvariety $T \subset V$ of dimension $\dim T \ge \d+3$ we have $e_T(B) \le r^2$.
This means that the set of points $Q \in V$ such that $e_Q(B) > r^2$
has dimension at most $\d+2$. Since $X$ is cut out by $\d+2$ general 
hyperplane sections of $V$ through $P$, 
it follows by Proposition~\ref{p:restr-mult} that
the set of points $Q \in X$ such that $e_Q(B|_X) > r^2$
is finite. Therefore we have $e_Q(B|_X) \le r^2$ for a general point $Q \in C$. 

Fix such a point $Q \in C$, and let $S \subset X$ be a smooth surface
cut out by general hyperplanes through $Q$. 
By applying again Proposition~\ref{p:restr-mult}, we see that $e_Q(B|_S) \le r^2$. 

Since $B|_S$ is a zero-dimensional complete intersection subscheme of $S$, 
the multiplicity $e_Q(B|_S)$ is computed by 
the Hilbert--Samuel multiplicity of the ideal $\I_{B|_S,Q} \subset \O_{S,Q}$
locally defining $B|_S$ near $Q$ (cf.~\cite[Exercise~7.1.10(a)]{Ful98}). 
Then \cite[Theorem~0.1]{dFEM04}
implies that the log canonical threshold of $(S,B|_S)$ near $Q$ satisfies the inequality
\[
\lct_Q(S,B|_S) \ge \frac{2}{\sqrt{e_Q(B|_S)}}.
\]
Since $e_Q(B|_S) \le r^2$ and $c < 1/r$, this implies that 
$\lct_Q(S,B|_S) > 2c$, and therefore
$(S,2cB|_S)$ is log terminal near $Q$. It follows by inversion of adjunction
that $(X,2cB|_X)$ is log terminal near $Q$. 
As $Q$ was chosen to be a general point of an arbitrary curve $C$ on $X$, 
we conclude that $(X,2cB|_X)$ is log terminal in dimension one. 
\end{proof}

Lemma~\ref{l:1} implies that the multiplier ideal $\J(X,2cB|_X)$
defines a zero-dimensional subscheme $\S \subset X$. 
We have $H^1(X,\J(X,2cB|_X) \otimes \O_X(\d+3)) = 0$ 
by Nadel's vanishing theorem (Theorem~\ref{t:Nadel}), and therefore there is a surjection
\[
H^0(X,\O_X(\d+3)) \surj H^0(\S,\O_\S(\d+3)) \cong H^0(\S,\O_\S)
\]
(here $\O_\S(\d+3) \cong \O_\S$ because $\S$ is zero dimensional). 
Keeping in mind that 
\[
H^0(X,\O_X(\d+3)) \cong H^0(\P^{N-\d-2},\O_{\P^{N-\d-2}}(\d+3)),
\]
it follows that
\begin{equation}
\label{eq:up1}
h^0(\S,\O_\S) \le h^0(X,\O_X(\d+3)) = \binom{N+1}{\d+3}.
\end{equation}

\begin{lem}
\label{l:2}
There is a prime divisor $E$ over $X$ with center $P$
and log discrepancy 
\begin{equation}
\label{eq:l:2}
a_E(X,cB|_X + (\d+2)P) \le 0,
\end{equation}
such that the center of $E$ on the blow-up of $X$ at $P$ has dimension $\ge \d + 2$. 
\end{lem}

\begin{proof}
Let $f \colon X' \to X$ be a log resolution of
$(X,B+P)$, and let $Y \subset X$ be a subvariety cut out by $\d+2$ general 
hyperplane sections through $P$. 
We remark that $\dim Y \ge 2$, given our assumption on $\d$ and $N$.
Let $Y' \subset X'$ be the proper transform of $Y$. By Bertini's theorem, we can ensure 
that $Y'$ intersects transversally the exceptional locus of $f$ 
(i.e., $Y'$ intersects transversally each stratum of the exceptional locus that it meets), 
and that the induced map $Y' \to Y$ is a log resolution of $(Y,B|_Y + P)$. 
By \eqref{eq:mld(X)} and 
inversion of adjunction, we have $\mld(P;Y,cB|_Y) \le 0$. This means that there is 
a prime exceptional divisor $F \subset Y'$ with center $P$ in $Y$ 
and log discrepancy $a_F(Y,cB|_Y) \le 0$. 
There is a unique prime exceptional divisor $E \subset X'$ such that $F$ is an irreducible
component of $E|_{Y'}$. Note that $E|_{Y'}$ is reduced. 
Since $E$ is the only prime divisor of $X'$ that is
contained in either supports of the inverse images of $B|_X$ and $P$
and whose restriction to $Y'$ contains $F$, we have 
$\val_E(B|_X) = \val_F(B|_Y)$ and $\val_E(P) = \val_F(P)$. 
Then the lemma follows by adjunction formula. 
\end{proof}

Let $E$ be as in Lemma~\ref{l:2}, and let
\[
\l := \frac{\val_E(P)}{c\val_E(B|_X)}.
\]

\begin{lem}
\label{l:lowerbound-lambda}
$(N+1)\l > 1$.
\end{lem}

\begin{proof}
For a general linear projection $\t \colon X \to \P^{N-\d-3}$, 
let $x_1,\dots,x_{N-\d-3} \in \fm_{X,P}$ be elements
obtained by pulling back a regular system of parameters of $\t(P) \in \P^{N-\d-3}$, 
and let $y_1,\dots,y_{\d+3} \in \fm_{X,P}$ be $\d+3$ general linear combinations
of these elements $x_1,\dots,x_{N-\d-3}$. Since $\d+3 \le N-\d-3$, 
we can assume that $y_1,\dots,y_{\d+3}$ are linearly independent. 

We claim that if $h(y_1,\dots,y_{\d+3})$ is 
any nonzero polynomial in these variables, then 
\[
\val_E(h) = \mult(h)\val_E(P).
\]
To see this, let $m := \mult(h)$, and let $h_m$ be the initial term
of degree $m$ of $h$. 
The center $C$ of $E$ in $\Bl_PX$ is contained in the exceptional divisor $E_P$ of $\Bl_PX \to X$.
Note that, by Lemma~\ref{l:2}, $C$ is a variety of dimension $\ge \d + 2$. 
By construction, there is a finite map from $E_P$
to the projective space $\Proj \C[x_1,\dots,x_{N-\d-3}]$, and 
linear projection (a rational map) from this space to $\Proj\C[y_1,\dots,y_{\d+3}]$.
If $y_1,\dots,y_{\d+3}$ are general, then $C$ dominates
$\Proj\C[y_1,\dots,y_{\d+3}]$, and therefore it cannot be contained in the
hypersurface defined by the equation $h_m(y_1,\dots,y_{\d+3}) = 0$ in $E_P$. 
Writing $E_P = \sum a_i E_i$, we have $\val_{E_i}(h) = ma_i$, and hence
\[  
\val_E(h) = \sum \val_{E_i}(h) \val_E(E_i) = \sum ma_i \val_E(E_i) = m \val_E(P),
\]
which proves our claim.

Suppose now that $d$ is a positive integer such that 
\[
d\val_E(P) \le - a_E(X,2cB|_X).
\]
The lemma implies that for every nonzero polynomial $h(y_1,\dots,y_{\d+3})$ 
of degree $\le d$ we have $\val_E(h) \le -a_E(X,2cB|_X)$, 
and therefore $h \not\in \J(X,2cB|_X)\.\O_{X,P}$. This means that if
$\cV \subset \O_{X,P}$ is the $\C$-vector space spanned by the polynomials
$h(y_1,\dots,y_{\d+3})$ of degree $\le d$, then the quotient map $\O_{X,P} \to \O_{\S,P}$
restricts to a injective map $\cV \inj \O_{\S,P}$. 
It follows that
\[
h^0(\S,\O_\S) \ge \dim_\C \cV = \binom{d+\d+3}{\d+3}.
\]
Comparing with \eqref{eq:up1}, 
we conclude that $d \le N-\d-2$, and hence we must have 
\[
(N-\d-1) \val_E(P) > - a_E(X,2cB|_X).
\]
Keeping in mind the definition of $\l$, this means that 
\[
a_E(X,(2-(N+1)\l)cB|_X + (\d+2)P) = a_E(X,2cB|_X - (N-\d-1)P) > 0.
\]
Since, on the contrary, we know that $a_E(X,cB|_X + (\d+2) P) \le 0$
by Lemma~\ref{l:2}, we conclude that $(N+1)\l > 1$, as stated. 
\end{proof}

Let $J_X \subset X$ be the subscheme defined by $\Jac_X$. 
Since $X$ has locally complete intersection singularities, 
we have $\^a_E(X) = a_E(X) + \val_E(J_X)$ by Remark~\ref{r:^k_E-lci}, and hence we have
\[
\^a_E(X,cB|_X + J_X + (\d+2)P) \le 0
\]
by \eqref{eq:l:2}.
By hypothesis, $V$ has a singularity of type $(\d,\n)$ at $P$.
It follows by Proposition~\ref{p:restr} that
$(\fm_{X,P})^{\n-1} \subset \ov{\Jac_X}$, 
and thus we have $(\n-1)\val_E(P) \ge \val_E(J_X)$. 
Therefore
\begin{equation}
\label{eq:^a_E}
\^a_E(X,cB|_X + (\d + \n + 1)P) \le 0.
\end{equation}
The inequality in \eqref{eq:^a_E} can be rewritten as follows:
\begin{equation}
\label{eq:^a_E-2}
\^a_E(X,(1-(N-2\d-\n-6)\l)cB|_X) \le (N-\d-5)\val_E(P).
\end{equation}
The next lemma implies that the pair in the right hand side of \eqref{eq:^a_E-2}
is effective. 

\begin{lem}
\label{l:upperbound-lambda}
$(N-2\d-\n-4)\l \le 1$.
\end{lem}

\begin{proof}
By the definition of $\l$ and Proposition~\ref{p:^lct}, we have
\[
\^a_E(X,(N-2\d-\n-4)\l cB|_X + (\d+\n+1)P ) = \^a_E(X,(N-\d-3)P) \ge 0.
\]
The assertion follows by contrasting this inequality with \eqref{eq:^a_E}.
\end{proof}

Let 
\[
\p \colon \P^{N-\d-2} \rat \P^{N-\d-4}
\]
be a very general linear projection. Let $Q := \p(Q)$ and
$A := \p_*[B|_X]$. Note that $A$ is a divisor on $\P^{N-\d-4}$ of degree $r^2N$. 
The divisorial valuation $\val_E$ restricts to a divisorial valuation $q\val_G$
with center $Q$ on $\P^{N-\d-4}$, where $q$ is a positive integer. 
By taking a general projection, we can ensure that $q\val_G(Q) = \val_E(P)$.
Since, by Lemma~\ref{l:upperbound-lambda}, the pair in the right hand side of \eqref{eq:^a_E-2}
is effective, and $N-\d-5 \ge 0$,  we can apply Theorem~\ref{t:projection}, which gives
\[
a_G\(\P^{N-\d-4},\frac{(1-(N-2\d-\n-6)\l)^2 c^2}{4}\, A\) \le (N-\d-5)\val_G(Q).
\]
Using again that we are dealing with effective pairs, we can apply inversion of adjunction. 
Thus, looking at the degree after restricting to a general line through $Q$, we conclude that
\[
\deg\(\frac{(1-(N-2\d-\n-6)\l)^2 c^2}{4}\, A\) \ge 1.
\]
Since $\deg A = r^2N$ and $(N+1)\l > 1$ (by Lemma~\ref{l:lowerbound-lambda}), we get
\[
rc > \frac{2(N+1)}{(2\d+\n+7)\sqrt{N}}.
\]
This completes the proof of Theorem~\ref{t:bound}. 
\end{proof}

\begin{proof}[Proof of Theorem~\ref{t:main}]
The inequality assumed in the theorem on the integers $\d,\n,N$
imply that $\d \le \frac N2 -3$, which in turns implies that $N-\d \ge 5$. 
Therefore $\d$ and $N$ satisfy the hypotheses of both Lemma~\ref{l:V}
and Theorem~\ref{t:bound}. In particular, we see that $V$ is factorial by Lemma~\ref{l:V}.
Adjunction shows that $\om_V \cong \O_V(1)$, 
hence $V$ is Fano. The Lefschetz hyperplane theorem
implies that the Picard group is generated by $\O_V(1)$.

The next step is to ensure that $V$ has terminal singularities. 
Suppose otherwise. Then there is a prime divisor $E$ over $V$ with
log discrepancy $a_E(V) \le 1$. Let $C \subset V$ be the center of $E$, and fix
a point $P \in C$. Note that $C$ has codimension $\ge 2$ in $V$ since $V$ is normal.
Let $s$ be a large enough integer so that the base locus
of the linear system $|\O_V(s)\otimes I_C|$ has codimension $\ge 2$, 
and let $Z \subset V$ be the subscheme cut out by two general members
of $|\O_V(s)\otimes I_C|$. Note that $\can_P(V,Z) \le 0$,
since $V$ is not terminal at $P$ and $P \in Z$. On the other hand, since $P$ is a singularity
of type $(\d,\n)$, Theorem~\ref{t:bound} implies that
\[
\can_P(V,Z) \ge \frac{2(N+1)}{r(2\d+\n+7)\sqrt N} > 0.
\]
This gives a contradiction, and therefore $V$ must have terminal singularities.

In particular, $V$ is a Mori fiber space (over a point), 
and it makes sense to inquire whether it is birationally superrigid. 
Suppose by contradiction that $V$ is not birationally superrigid.
Then there is a birational map $\f \colon V \rat V'$  
from $V$ to a Mori fiber space $V'$ that is not an isomorphism. 
Fix a projective embedding $V' \subset \P^m$, and let
$\cH = \f_*^{-1}|\O_{V'}(1)|$. Note that
$\cH \subset |\O_V(r)|$ for some integer $r \ge 1$.
Let $B \subset V$ be the intersection of two general members of $\cH$. 
The Noether--Fano inequality (Theorem~\ref{t:NF}) implies that the pair
$(V,B)$ has canonical threshold
\[
\can(V,B) < \frac 1r.
\]
Then Theorem~\ref{t:bound} implies that
\[
\can(V,B) \ge \frac{2(N+1)}{r(2\d+\n+7)\sqrt N}.
\]
By comparing these two inequalities, we obtain
\[
2\d+\n + 7 > \frac{2(N+1)}{\sqrt N},
\]
in contradiction with our assumptions.
We conclude that $V$ is birationally superrigid.
\end{proof}

\end{document}